\documentclass[11pt,a4paper]{amsart} 
\usepackage{amssymb,latexsym} 
\usepackage[arrow,matrix,all,cmtip]{xy}
\usepackage[all,cmtip]{xy}
\usepackage{marginnote}
 
 \newcommand\myarat{\ar@}

\numberwithin{equation}{section} 
\newtheorem{theorem}{Theorem}[section] 
\newtheorem{proposition}[theorem]{Proposition} 
\newtheorem{corollary}[theorem]{Corollary} 
\newtheorem{definition}[theorem]{Definition} 
 
\newtheorem{remark}[theorem]{Remark} 
\newtheorem{lemma}[theorem]{Lemma} 
 
\newtheorem{example}[theorem]{Example}

\newtheorem{problem}[theorem]{Problem}

\def\Hom{\operatorname{Hom}}

\usepackage{amsmath}	
\begin{document} 
 
\title{The motivic fundamental group of a punctured elliptic curve and algebraic cycles} 
 
\author{Jin Cao and Tomohide Terasoma} 

\address{
Jin Cao\\
Yau Mathematical Sciences Center,
Tsinghua University, Beijing, China
}
\email{caojin@mail.tsinghua.edu.cn}

\address{
Tomohide Terasoma\\
Faculty of Science and Engineering, 
Hosei University, Koganei, Tokyo, Japan
}
\email{terasoma@hosei.ac.jp }

\date{\today} 
 
\maketitle 
\markright{}
 
\makeatletter 

\begin{abstract}
In this paper, we consider the motivic fundamental group of the punctured elliptic curves as a DG complex in the DG category of elliptic motives and describe its resolution via Schur complexes. During this process, we find the algebraic cycles analogous to the Bloch-Totaro cycles.
\end{abstract}

\tableofcontents
\section{Introduction}
The algebraic fundamental group of $\mathbb{P}^1 - \{0, 1, \infty\}$ becomes one of the most interesting objects in arithmetic algebraic geometry after the suggestion of Grothendieck \cite{D}. In fact, the classical Belyi's theorem asserts that the canonical action of the absolute Galois group $\mathrm{Gal}(\overline{\mathbb{Q}} / \mathbb{Q})$ on the profinite completion of the fundamental group of $\mathbb{P}^1 - \{0, 1, \infty\}$ is faithful. This motivated Grothendieck to study the absolute Galois group by the aid of the fundamental group of $\mathbb{P}^1 - \{0, 1, \infty\}$. Further exploring these ideas, Deligne \cite{D} initiated the study of this fundamental group based on the theory of motives. From the motivic point view, it has been known that the correct object to study should be the nilpotent completion of the fundamental group and then one may consider its motivic version. Deligne and Goncharov \cite{DG} studied this motivic fundamental group and further derived interesting information of multiple zeta values, which are the periods of the motivic fundamental group of $\mathbb{P}^1 - \{0, 1, \infty\}$. A priority, the motivic fundamental group of $\mathbb{P}^1 - \{0, 1, \infty\}$ is an object in the category of mixed Tate motives. For the different constructions of mixed (Tate) motives, we refer to \cite{L05}. Undoubtedly, each construction has its advantages. If we assume that the base field satisfied the Beilinson-Soul\'{e} vanishing conjectures, then all the K-theoretical or algebraic cycle theoretical constructions of mixed Tate motives coincide, for example, Bloch-Kriz-May's construction \cite{BK,KM}, Levine's construction and even Voevodsky's construction \cite{L10}. Then it is natural to search a good understanding the motivic fundamental group of $\mathbb{P}^1 - \{0, 1, \infty\}$ in these categories.
\

In  \cite{G11}, Guillou described a construction of the motivic fundamental group of $\mathbb{P}^1 - \{0, 1, \infty\}$ in the Bloch-Kriz category of mixed Tate motives. The rough idea is as follows: Following Beilinson, Deligne and Goncharov, there is a topological construction of (n-truncated) nilpotent completion of a topological space $X$ under suitable conditions on $X$. One may mimick this construction in the motivic world by choosing a suitable complex, where replaces the topological complexes by Bloch's cycle complexes or their variants. In particular, in Bloch-Kriz-May's category of mixed Tate motives, Guillou tried to find out a cellular approximation of this complex in the case of $X = \mathbb{P}^1 - \{0, 1, \infty\}$. During this process, Totaro's cycle \cite{Tot} and their generalizations -- Bloch-Totaro cycles appear naturally. 
\

On the other hand, there are natural generalizations of Bloch-Kriz-May's construction in the case of elliptic curves, see \cite{C, KT, P}, i.e., a reasonable (triangulated) category of mixed motives for an elliptic curve. However the Beilinson-Soul\'{e} vanishing conjectures are unknown in this case. Anyway one may ask: 
\textbf{
\begin{problem}
How to construct the motivic fundamental group of a punctured elliptic curve in the category of mixed elliptic motives like the case of $\mathbb{P}^1 - \{0, 1, \infty\}$? Could we find some generalization of Bloch-Totaro cycles in the elliptic case? 
\end{problem}
}

In this paper, we fix an elliptic curve $E$ without complex multiplication. The paper is organized as follows. In the second section, we briefly review the triangulated (DG) category of motives of an elliptic curve. Moreover we will compute some concrete boundary cycles appearing in the decomposition of elliptic motives, which are crucial for the further application. In the third section, we consider linear cycles, which are firstly defined in \cite{BL}, together with their decompositions, see Proposition \ref{dec of linear cycles}. Moreover we provide an explicit decomposition in the following section. In the fifth section, we consider the resolution of DG complexes, which is similar to the resolution of cell modules in the sense of \cite{KM, C}.  As an application, we consider the nilpotent completion of the fundamental group of a punctured elliptic curve in the motivic world.

\section{Motives of an elliptic curve}
\subsection{Description of elliptic motives via DG complexes}
In this subsection, we will give a very short introduction to motives of an elliptic curve via the language of DG complexes. 
\

We let $\mathbf{DM}_{\mathrm{gm}}(k)$ be Voevodsky's triangulated category of geometric motives.
For a projective smooth variety $X$, we can associate a motive $M(X)$ of $X$ in $\mathbf{DM}_{\mathrm{gm}}(k)$. 
\

We define
\[
\mathbf{R}\Gamma_M(X)=M(X)^*=M(X)(-d_X)[-2d_X]
\]
and for any integer $p$, we let
\[
\underline{\Hom}^{-p}_M(\mathbf{R}\Gamma_M(Y),\mathbf{R}\Gamma_M(X))
=Z^{d_Y}(X \times Y,p).
\]
Then we can use these data to define a (quasi)-DG category\footnote{Here "quasi" means that the composition of internal homomorphism is well-defined up to quasi-isomorphism. However if one insists the Friedlander-Suslin construction, one may drop it.} $\widetilde{\mathbf{DM}}_{\mathrm{gm}}(k)$, whose homotopy category is $\mathbf{DM}_{\mathrm{gm}}(k)$.
If we associate $f: X \to Y$ with 
\[
f^* = \ ^t \Gamma_f \in \underline{\Hom}_M^0(\mathbf{R}\Gamma_M(Y),\mathbf{R}\Gamma_M(X)),
\]
then we will get a contravariant functor:
\[
\mathbf{R}\Gamma_M:(\mathbf{SmProj})^{op}\to \widetilde{\mathbf{DM}}_{\mathrm{gm}}(k)
\]
The functoriality of $\mathbf{R}\Gamma_M$ comes from the composition law of correspondences if we view $f^*$ and $g^*$ as correspondences, where $f: X\to Y$ and $g:Y \to Z$.

For an idempotent 
$\epsilon_X\in End(\mathbf{R}\Gamma_M(X))$
and 
$\epsilon_Y\in End(\mathbf{R}\Gamma_M(Y)), $
we have:
\[
\underline{\Hom}^{-p}(\epsilon_Y\mathbf{R}\Gamma_M(Y),\epsilon_X\mathbf{R}\Gamma_M(X))
=\epsilon_XZ^{d_Y}(X \times Y,p)\epsilon_Y
\]
and more generally,
\[
\underline{\Hom}^{-p}(\epsilon_Y\mathbf{R}\Gamma_M(Y),\epsilon_X\mathbf{R}\Gamma_M(X)(m)[2m])
=\epsilon_XZ^{d_Y+m}(X \times Y,p)\epsilon_Y
\]
\begin{definition}
We denote the DG category generated by the motives of an elliptic curve $E$ by $\widetilde{\mathbf{DMEM}}(k, \mathbb{Q})_E$.   
\end{definition}

\begin{example}
Let $E$ be an elliptic curve. We set
$$
V=\mathbf{R}\Gamma_M(E)^{(-)}[1]=\epsilon\mathbf{R}\Gamma_M(E)[1]
$$
with
\begin{equation} \label{1-comp}
\epsilon=\frac{1}{2}[(x,x)-(x,-x)]\in Z^1(E\times E).
\end{equation}
\end{example}

We recall that the main result in \cite{C} (Theorem 8.13), which says that: 
\begin{theorem}
There is a cycle algebra $\mathcal{E}_{ell}$, which is called the elliptic cycle of $E$ such that the triangulated category consists of compact objects in $\mathcal{D}_{\mathcal{E}_{ell}}^{GL_2}$ is equivalent to the triangulated subcategory $\mathbf{DMEM}(k, \mathbb{Q})_E$ of Voevodsky's motives generated by the motives of $E$.  Here $\mathcal{D}_{\mathcal{E}_{ell}}^{GL_2}$ is the derived category of dg $\mathcal{E}_{ell}$-modules.
\end{theorem}
The DG category of cell $\mathcal{E}_{ell}$-modules provides a natural dg enhancement of $\mathbf{DMEM}(k, \mathbb{Q})_E$. For any DG category, the second author introduces the notions of DG complexes in a DG category in \cite[Subsection 2.2]{T10}.
\begin{definition}
A sequence $K = \{M_i, f_{ij}\}$, where $M_i \in \widetilde{\mathbf{DMEM}}(k, \mathbb{Q})_E$, is called a DG complex in $\widetilde{\mathbf{DMEM}}(k, \mathbb{Q})_E$ if
\begin{enumerate}
\item
$f_{ij} \in \underline{Hom}_{\widetilde{\mathbf{DMEM}}(k, \mathbb{Q})_E}^{1+j-i}(M_j, M_i)$ for $i > j$\footnote{$\underline{Hom}$ denote the internal homorphism in $\widetilde{\mathbf{DMEM}}(k, \mathbb{Q})_E$.};
\item
$\partial(f_{ik}) + \sum_{i > k >j} f_{ij} \circ f_{jk} = 0$.
\end{enumerate}
Then the category of DG complexes in $\widetilde{\mathbf{DMEM}}(k, \mathbb{Q})_E$ forms a (quasi-)DG category.
\end{definition}
\begin{remark}
One may equivalently consider the DG complexes in the DG category of dg $\mathcal{E}_{ell}$-modules due to the above theorem. We don't distinguish these two DG categories and DG categories of DG complexes. We denote it by $K\mathcal{C}_{\mathcal{E}_{ell}}$.
\end{remark}
\begin{remark}
Based on the construction of simplicial bar construction, the second author also defines the following equivalences between DG categories: 
\begin{itemize}
\item{(Theorem 6.4 in \cite{T10})}
The DG category $K\mathcal{C}_{A}$ of finite DG complexes in the DG category of dg $A$-modules is homotopy equivalent to the DG category of bounded comodules over the simplicial bar complex $B_{simp}(\epsilon | A | \epsilon)$, where the morphisms of latter category need a little modification.
\end{itemize}
One can similarly define the simplical bar complexes over $GL_2$ by requiring each term appearing in the definitions of \cite[Section 4]{T10} equipped with a $GL_2$-action. Then these constructions can be applied to the following case: The DG category $K\mathcal{C}_{\mathcal{E}_{ell}}$ of finite DG complexes in the category of elliptic motives is homotopy equivalent to the DG category of $B_{simp}(\mathcal{E}_{ell},\epsilon)$-comodules or even the DG category of $B(\mathcal{E}_{ell},\epsilon)$-comodules. Due to these equivalences, we will only focus on the structure of the motivic fundamental group of a punctured elliptic curve as a DG complex.
\end{remark}

\subsection{Decomposition of elliptic motives and algebraic cycles}
We recall the definition of the theta function $\theta_{1,1}(z,\tau)$  in \cite{Igusa}, which is defined as
\[
\sum_{n \in \mathbb{Z}}\exp\{2\pi i [\frac{\tau}{2}(n+\frac{1}{2})^2 + (n+\frac{1}{2})(z+\frac{1}{2})]\},
\]
where $z \in \mathbb{C}, \tau \in \mathbb{H}$, the upper half plane. We let $L$ be the lattice generated by $1$ and $\tau$. Then $\theta_{1,1}(z,\tau)$ has exactly one zero in $\mathbb{C}/L$, which is the original point $0$ in $\mathbb{C}/L$. For simplicity, we denote it by $\theta$.
\begin{theorem} \label{theta-theorem}
Fix positive integers $n, a_i(i = 1, \cdots, m)$, $b_j(j = 1, \cdots, l)$. We let $f_i(x_1, \cdots, x_n)$ and $ g_j(x_1, \cdots, x_n)$ be the degree $1$ homogeneous polynomials in $x_1, \cdots, x_n$ with integer coefficients. If
\begin{equation} \label{condition for theta}
\begin{split}
& a_1 f_1^2(x_1, \cdots, x_n) + \cdots + a_m f_m^2(x_1, \cdots, x_n) \\
= & b_1 g_1^2(x_1, \cdots, x_n) + \cdots + b_l g_l^2(x_1, \cdots, x_n),
\end{split}
\end{equation}
then
\begin{equation}
\frac{\theta(f_1(x_1, \cdots, x_n))^{a_1} \theta(f_2(x_1, \cdots, x_n))^{a_2}  \cdots \theta(f_m(x_1, \cdots, x_n))^{a_m}}{\theta(g_1(x_1, \cdots, x_n))^{b_1} \theta(g_2(x_1, \cdots, x_n))^{b_2} \cdots \theta(g_l(x_1, \cdots, x_n))^{b_l}}
\end{equation}
is a rational function on $E^n \cong (\mathbb{C}/L)^n$.
\end{theorem}
\begin{proof}
According to the transformation law in \cite[P123]{Igusa}, we have:
\[
\theta(z+p\tau) = (-1)^{-p} \exp[2\pi i(-p^2\frac{\tau}{2} - pz)] \theta(z)
\]
for $p \in \mathbb{Z}$. We let
\[
f_i(x_1,x_2, \cdots, x_n) = \sum_{ r= 1}^n c_{ir} x_r
\]
and
\[
g_j(x_1, x_2, \cdots, x_n) = \sum_{s= 1}^n d_{js} x_s.
\]
After changing the coordinates
\[
(x_1, x_2, \cdots, x_n) \to (x_1 + \tau, x_2, \cdots, x_n),
\]
we have:
\[
\theta(f_1(x_1, \cdots, x_n)) \to (-1)^{-c_{11}} \exp[2\pi i (-c_{11}^2 \frac{\tau}{2} - c_{11}f_1)] \theta(f_1(x_1, \cdots, x_n)).
\]
In general, we have:
\[\theta(f_1(x_1, \cdots, x_n))^{a_1} \theta(f_2(x_1, \cdots, x_n))^{a_2}  \cdots \theta(f_m(x_1, \cdots, x_n))^{a_m}\]
becomes
\begin{equation} \nonumber
\begin{split}
&(-1)^{-(a_1c_{11}+\cdots+a_mc_{m1})} \exp\{2\pi i [-(a_1c_{11}^2+\cdots+a_mc_{m1}^2) \frac{\tau}{2} - (a_1c_{11}f_1+\cdots+ \\ 
&a_mc_{m1}f_m)]\} 
\theta(f_1(x_1, \cdots, x_n))^{a_1} \theta(f_2(x_1, \cdots, x_n))^{a_2}  \cdots \theta(f_m(x_1, \cdots, x_n))^{a_m}
\end{split}
\end{equation}
We have a similar formula for 
\[\theta(g_1(x_1, \cdots, x_n))^{b_1} \theta(g_2(x_1, \cdots, x_n))^{b_2} \cdots \theta(g_l(x_1, \cdots, x_n))^{b_l}.\]
Note that the assumption (\ref{condition for theta}) implies that:
\begin{itemize}
\item
$a_1c_{11}^2+\cdots+a_mc_{m1}^2 = b_1d_{11}^2+\cdots+b_ld_{l1}^2$ if we compare the coefficients of $x_1^2$ in LHS of  (\ref{condition for theta}) with its RHS;
\item
Similarly we can get $a_1c_{11}f_1+\cdots +a_mc_{m1}f_m = b_1d_{11}g_1+\cdots +b_ld_{l1}g_l$ if we compare the coefficients of $x_1x_i$ between LHS and RHS of (\ref{condition for theta}) and sum these coefficients together.
\end{itemize}
Note that $a_p c_{pq}^2 \equiv a_p c_{pq} \ \mathrm{mod} \ 2$. The equality
\[a_1c_{11}^2+\cdots+a_mc_{m1}^2 = b_1d_{11}^2+\cdots+b_ld_{l1}^2\]
implies that
\[
a_1c_{11}+\cdots+a_mc_{m1} \equiv b_1d_{11}+\cdots+b_ld_{l1} \ \mathrm{mod} \ 2.
\]
Hence for any $1 \leq t \leq n$, we have:
\[(-1)^{-(a_1c_{1t}+\cdots+a_mc_{mt})} = (-1)^{-(b_1d_{1t}+\cdots+b_ld_{lt})}.\]
Therefore for any $\tau \in L$, 
\[
\frac{\theta(f_1(x_1, \cdots, x_n))^{a_1} \theta(f_2(x_1, \cdots, x_n))^{a_2}  \cdots \theta(f_m(x_1, \cdots, x_n))^{a_m}}{\theta(g_1(x_1, \cdots, x_n))^{b_1} \theta(g_2(x_1, \cdots, x_n))^{b_2} \cdots \theta(g_l(x_1, \cdots, x_n))^{b_l}}
\]
is invariant under the transformation $(x_1, \cdots, x_n) \to (x_1+\tau, \cdots, x_n)$. Similarly, this implies that
\[
\frac{\theta(f_1(x_1, \cdots, x_n))^{a_1} \theta(f_2(x_1, \cdots, x_n))^{a_2}  \cdots \theta(f_m(x_1, \cdots, x_n))^{a_m}}{\theta(g_1(x_1, \cdots, x_n))^{b_1} \theta(g_2(x_1, \cdots, x_n))^{b_2} \cdots \theta(g_l(x_1, \cdots, x_n))^{b_l}}
\]
is invariant under the action of $L^n$. Hence it is a rational function on $E^n$.
\end{proof}

Motivated by Lemma 4.2 in \cite{KT}, we have the following observations. We will also follow the notions used in Section 4.1 of \cite{KT}. Recall that we let $\pi_i: E^n \to E$ and $\pi_{ij}: E^n \to E \times E$ be the projection to $i$-th and $j$-th components, and the diagonal (resp. anti-diagonal) in $E \times E$ is denoted by $\Delta$ (resp. $\Delta^{-}$). The classes $\pi^*_i([0])$, $\pi^*_{ij}(\Delta)$ and $\pi^*_{ij}(\Delta^{-})$ in $Z^1(E^n)$ are denoted by $P_i$, $\Delta_{ij}$ and $\Delta_{ij}^{-}$ respectively. We set
\[
D_{ij} = - \Delta_{ij}+P_i+P_j, D_{ij}^{-} = - \Delta_{ij}^{-}+P_i+P_j,
\]
for $i \neq j$. Hence $D_{ij} = D_{ji}$. Via the map $E \times E \xrightarrow{id \times \sigma} E \times E, (x, y) \to (x, -y)$, we have $(id \times \sigma)^*(D_{12}) = - D_{12} = D_{12}^{-}$ in $CH^2(E^2)$.
\begin{proposition}
We have a decomposition 
\[
\mathbf{R}\Gamma_M(E) \cong \mathbb{Q} \oplus V[-1] \oplus \mathbb{Q}(-1)[-2].
\]
More precisely, we have two morphisms:
\[
\mathbb{Q} \oplus V[-1] \oplus \mathbb{Q}(-1)[-2] \xrightarrow{\beta} \mathbf{R}\Gamma_M(E) \xrightarrow{\alpha}  \mathbb{Q} \oplus V[-1] \oplus \mathbb{Q}(-1)[-2]
\]
such that 
\begin{itemize}
\item
$\alpha \beta = id$;
\item
there exists an algebraic cycle $W_{\mathbf{R}\Gamma_M(E)}$ such that
\[
\partial W_{\mathbf{R}\Gamma_M(E)} = \beta \alpha - id_{\mathbf{R}\Gamma_M(E)}.
\]
\end{itemize}
\end{proposition}
\begin{proof}
Note that 
$\beta_1 := \pi^*$, where $\pi: E \to Spec(k)$ is the structure map, defines a morphism from $\mathbb{Q} = \mathbf{R}\Gamma_M(Spec(k))$ to $\mathbf{R}\Gamma_M(E)$. $\beta_2 = \epsilon$, defined in (\ref{1-comp}), induces a morphism between $V[-1]$ and $\mathbf{R}\Gamma_M(E)$. Let $\beta_3$ be the transpose of the graph $i_0$, where $i_0: \{0\} \to E$, and then we define $\beta = \beta_1 \oplus \beta_2 \oplus \beta_3$. Similarly we define $\alpha$ to be the sum of the graph of the morphism $i_0$, $\epsilon$ and the transpose $\pi^*$. Then one may check that $\alpha \beta$ induces an identity map on $\mathbb{Q} \oplus V[-1] \oplus \mathbb{Q}(-1)[-2]$. Furthermore we have: 
\[
\beta \alpha = E \times \{0\} + \frac{1}{2}(\Delta - \Delta^-) + \{0\} \times E \in Z^1(E \times E).
\]
Hence we have
\begin{equation}
\begin{split}
& \beta \alpha - id \\
= & E \times \{0\} + \frac{1}{2}(\Delta - \Delta^-) + \{0\} \times E - \Delta \\
= & E \times \{0\} - \frac{1}{2}(\Delta + \Delta^-) + \{0\} \times E \in Z^1(E \times E)
\end{split}
\end{equation}
It is easy to check the function $\frac{\theta(x+y)\theta(x-y)}{\theta(x)^2 \theta(y)^2}$ on $(x, y) \in E \times E$ satisfies the condition in the Theorem \ref{theta-theorem}, which is a rational function on $E \times E$. Moreover, 
\[
\mathrm{div}(\frac{\theta(x+y)\theta(x-y)}{\theta(x)^2 \theta(y)^2}) = (\Delta + \Delta^-) - 2(E \times \{0\} + \{0\} \times E).
\]
Therefore we can find an algebraic cycle $W_{\mathbf{R}\Gamma_M(E)}$ such that
\[
\partial W_{\mathbf{R}\Gamma_M(E)} = \beta \alpha - id_{\mathbf{R}\Gamma_M(E)}.
\]
\end{proof}
\begin{remark}
In the above proof, we show that there exists an algebraic cycle $W_{V12}$ -- the half of the graph of the function $\frac{\theta(x+y)\theta(x-y)}{\theta(x)^2 \theta(y)^2}$ such that $\partial W_{V12} = - D_{12} - (id \times \sigma)^* D_{12}$. Here the lower index $12$ will be useful in the future.
\end{remark}
\begin{proposition}
We have a decomposition 
\[
\wedge^2 V \cong  \mathbb{Q}(-1).
\]
More precisely, we have two morphisms:
\[
\mathbb{Q}(-1) \xrightarrow{\beta} \wedge^2 V \xrightarrow{\alpha} \mathbb{Q}(-1)
\]
such that 
\begin{itemize}
\item
$\alpha \beta =  id$;
\item
there exists an algebraic cycle $W_{\wedge^2 V}$ such that
\[
\partial W_{\wedge^2 V} = \beta \alpha - id_{\wedge^2 V}.
\]
\end{itemize}
\end{proposition}
\begin{proof}
In order to compute the composition law explicitly, we let $\beta$ be the algebraic cycle $-\frac{1}{4}(\Delta - \Delta^-) \in Z^1(E_1 \times E_2)$ and $\alpha$ be the algebraic cycle $\frac{1}{2}(\Delta - \Delta^-) \in Z^1(E_3 \times E_4)$, where $E_i (i = 1, \cdots, 4)$ are both the copy of $E$. Then we may compute $\alpha \beta$ firstly. This is done via the direct computation of intersection points $\Delta^-_{34} \cap \Delta_{12}$ and $\Delta_{34} \cap \Delta^-_{12}$, which are exactly 8 points. 
\

Next we can compute $\beta \alpha - id_{\wedge^2 V}$. By our definition, this algebraic cycle is:
\begin{equation} \nonumber
\begin{split}
& - \frac{1}{8}[(\Delta_{12} - \Delta^-_{12}) \times (\Delta_{34} - \Delta^-_{34})] - \frac{1}{8}[(\Delta_{13} - \Delta^-_{13}) \times (\Delta_{24} - \Delta^-_{24}) \\
&+ (\Delta_{14} - \Delta^-_{14}) \times (\Delta_{23} - \Delta^-_{23})] 
\end{split}
\end{equation}
Using these $D_{ij}$, then we have:
\begin{equation} \nonumber
\begin{split}
& \beta \alpha - id_{\wedge^2 V} = - \frac{1}{8}[(D_{12} - D^-_{12}) \times (D_{34} - D^-_{34})] \\
 &- \frac{1}{8}[(D_{13} - D^-_{13}) \times (D_{24} - D^-_{24}) 
+ (D_{14} - D^-_{14}) \times (D_{23} - D^-_{23})] \\
=&  - \frac{1}{8}[(D_{12} - (id \times \sigma)^*(D_{12})) \times (D_{34} -  (id \times \sigma)^*D_{34})] \\
 &- \frac{1}{8}[(D_{13} -  (id \times \sigma)^*D_{13}) \times (D_{24} -  (id \times \sigma)^*D_{24}) \\
&+ (D_{14} -  (id \times \sigma)^*D_{14}) \times (D_{23} -  (id \times \sigma)^*D_{23})] \\
= & -\frac{1}{8}[(2D_{12} + \partial W_{V12}) \times (2D_{34} + \partial W_{V34}) \\ & +(2D_{13} 
+ \partial W_{V13}) \times (2D_{24} + \partial W_{V24}) \\
&+(2D_{14} + \partial W_{V14}) \times (2D_{23} + \partial W_{V23}) ] \\
=& -\frac{1}{2}[D_{12} \times D_{34} + D_{13} \times D_{24}+D_{14} \times D_{23}] + \partial \widetilde{W}_1.
\end{split}
\end{equation}
Lemma 4.2 in \cite{KT} shows that $D_{12} \times D_{34} + D_{13} \times D_{24}+D_{14} \times D_{23} = 0$ in $CH^2(E^4)$ and in the following we'd like to find out the explicit boundary.
\

Via the proof of Lemma 4.2 in loc.cit, the following identities between algebraic cycles in the Chow group hold:
\begin{itemize} 
\item
$P_iD_{ij} = 0;$
\item
$-P_2D_{13} +D_{12}D_{13} = - P_1D_{23} + D_{12}D_{23}.$
\end{itemize}
Applying the automorphism $(x_1, x_2,x_3) \to (x_1, -x_2, x_3)$ to the last identity, we get:
\[-\sigma_2^*(P_2)\sigma_2^*(D_{13}) +\sigma_2^*(D_{12})\sigma_2^*(D_{13}) = - \sigma_2^*(P_1)\sigma_2^*(D_{23}) + \sigma_2^*(D_{12})\sigma_2^*(D_{23}).\] Hence we have:
\begin{equation}
\begin{split}
&-P_2D_{13} - (D_{12} + \partial W_{V12}) D_{13} \\
= & P_1(D_{23} + \partial W_{V23})+(D_{12} + \partial W_{V12})(D_{23} + \partial W_{V23}).
\end{split}
\end{equation}
In other words, there exists an algebraic cycle $\widetilde{W}_2$ such that:
\begin{equation} \nonumber
\begin{split}
-P_2D_{13} - D_{12} D_{13} 
=  P_1D_{23}+D_{12}D_{23} - \partial \widetilde{W}_2.
\end{split}
\end{equation}
Therefore we get 
\begin{equation} \label{cycle2}
P_2D_{13} = - D_{12}D_{23} +  \frac{1}{2}\partial \widetilde{W}_2.
\end{equation}
For the later use, we denote the algebraic cycle, whose boundary is $P_1 D_{23} + D_{12}D_{13}$ by $\widetilde{W}_{V123}$, i.e.,
\begin{equation}\label{V123}
\partial \widetilde{W}_{V123} = P_1 D_{23} + D_{12}D_{13}.
\end{equation}
We consider the map $\varphi: E_1 \times E_2 \times E_3 \times E_4 \to E_1 \times E_2 \times E_3$, which is defined by
$(x_1, x_2,x_3, x_4) \to (x_1, x_2 - x_4,x_3)$. Applying $\varphi$ to the equality (\ref{cycle2}), we have
\begin{equation}\label{eq}
\varphi^*(P_2) \varphi^* (D_{13}) = - \varphi^* (D_{12}) \varphi^* (D_{23}) +  \varphi^* \frac{1}{2}\partial \widetilde{W}_2.
\end{equation}
Note that $\varphi^*(P_2) = P_2+P_4 -D_{42}, \varphi^* (D_{13}) = D_{13}$.
\

Claim: There exists an algebraic cycle $W_{V124}$ such that 
\[
\varphi^*(D_{12}) = D_{12} - D_{14} + \partial W_{V124}.
\]
Consider the morphism $\varphi: E_1 \times E_2 \times E_4 \to E_1 \times E_2$ defined by $(x_1, x_2,x_4) \to (x_1, x_2 - x_4)$. So $\varphi^*(\Delta_{12}) = \{(x_1, x_2,x_4) \in  E_1 \times E_2 \times E_4 \mid x_1 - x_2 +x_4 = 0\}$. Then we consider the function
\[
f_{124} = \frac{\theta(x_1-x_2)\theta(x_2-x_4)\theta(x_1)\theta(x_4)}{\theta(x_1-x_2+x_4) \theta(x_1-x_4)\theta(x_2)}
\]
on $(x_1, x_2,x_4) \in  E_1 \times E_2 \times E_4$, which is  a rational function via Theorem \ref{theta-theorem}. The boudary of its graph is just:
\[
 (\Delta_{12} + \Delta_{24} +P_1+P_4) - (\varphi^*(\Delta_{12}) + \Delta_{14} +P_2). 
\] 
Rearranging these terms, it is the same as
\[
\varphi^*(D_{12}) - D_{12} + D_{14}. 
\]
Hence we may take the graph of the function $\pi_{124}^*f_{124}$ as the algebraic cycle $W_{V124}$, where $\pi_{124}: E_1 \times E_2 \times E_3 \times E_4 \to E_1 \times E_2 \times E_4$.
\

Similarly there exists an algebraic cycle $W_{V234}$ such that 
\[
\varphi^*(D_{23}) = D_{23} - D_{34} + \partial W_{V234}.
\]
Hence equality (\ref{eq}) implies that:
\begin{equation} \nonumber
\begin{split}
& \varphi^*(P_2)  \varphi^* (D_{13}) = (P_2+P_4 -D_{42}) D_{13} \\
= & - \varphi^* (D_{12}) \varphi^* (D_{23}) +  \varphi^* (\frac{1}{2}\partial \widetilde{W}_2) \\
=& -(D_{12} - D_{14} + \partial W_{V124})(D_{23} - D_{34} + \partial W_{V234}) + \frac{1}{2}\partial  \varphi^*\widetilde{W}_2.
\end{split}
\end{equation}
This above equation can be simplified as
\begin{equation} \nonumber
(P_2 + P_4 -D_{42})D_{13} = - (D_{12} - D_{14})(D_{23}-D_{34}) + \partial \widetilde{W}_3,
\end{equation}
Combining with $P_2D_{13} = -D_{12}D_{23} + \frac{1}{2} \partial \widetilde{W}_2$ and $P_4D_{13} = -D_{14}D_{43} + \frac{1}{2} \partial \widetilde{W}_4$, we find an algebric cycle $\widetilde{W}_{V1234}$ such that
\begin{equation} \label{V1234}
\partial \widetilde{W}_{V1234} = D_{12} \times D_{34} + D_{13} \times D_{24}+D_{14} \times D_{23} = 0.
\end{equation}
Hence there exists an algebraic cycle $W_{\wedge^2 V}$ such that
\begin{equation} \nonumber
\partial W_{\wedge^2 V} = -\frac{1}{2}[D_{12} \times D_{34} + D_{13} \times D_{24}+D_{14} \times D_{23}] + \partial \widetilde{W}_1 = \beta \alpha - id_{\wedge^2 V}.
\end{equation}
\end{proof}

\begin{corollary}
For $i \in \mathbb{Z}, i \geq 3$, we have a decomposition 
\[
\wedge^i V \cong  0.
\]
More precisely, there exists an algebraic cycle $W_{\wedge^i V}$ such that $\partial W_{\wedge^i V} = id_{\wedge^i V}.$
\end{corollary}
\begin{proof}
Assume that $i = 3$ first. According to the definition of $\wedge^3 V$, we have:
\[
id_{\wedge^3 V} = \sum_{s} D_{1s(1)} \times D_{2s(2)} \times D_{3s(3)},
\] 
where the index set runs over all bijections $s: \{1,2,3\} \to \{4,5,6\}$.
\

Using the results in the previous proof, we find that
\[
D_{14} \times D_{25} \times D_{36} + D_{15} \times D_{24} \times D_{36} = -D_{12} \times D_{45} \times D_{36} + \partial W_{V1245} \times D_{36};
\]
\[
D_{14} \times D_{26} \times D_{35} + D_{16} \times D_{24} \times D_{35} = -D_{12} \times D_{46} \times D_{36} + \partial W_{V1246}\times D_{35};
\]
\[
D_{15} \times D_{26} \times D_{34} + D_{16} \times D_{25} \times D_{34} = -D_{12} \times D_{56} \times D_{34} + \partial W_{V1256} \times D_{34}.
\]
Summing up the above three identities, we get that:
\begin{equation} \nonumber
\begin{split}
& id_{\wedge^3 V} = -D_{12} \times D_{45} \times D_{36} + \partial W_{V1245} \times D_{36} -D_{12} \times D_{46} \times D_{36} \\
& + \partial W_{V1246}\times D_{35} -D_{12} \times D_{56} \times D_{34} + \partial W_{V1256} \times D_{34} \\
& = \partial W_{V1245}\times D_{36} + \partial W_{V1246}\times D_{35} + \partial W_{V1256}\times D_{34} - D_{12}\times \partial W_{V3456}.
\end{split}
\end{equation}
If $i > 3$, one may proceed by induction. Assume $i = n$ and there exists an algebraic cycle such that $\partial W_{\wedge^n V} = id_{\wedge^n V}$.  Then for $i = n+1$, we have:
\[
id_{\wedge^{n+1} V} = \sum_{s} D_{1s(1)} \times D_{2s(2)} \cdots \times D_{(n+1)s(n+1)},
\] 
where the index set runs over all bijections $s: \{1,2,\cdots, n+1\} \to \{n+2,\cdots,2(n+1)\}$. Then we may divide $\sum_{s} D_{1s(1)} \times D_{2s(2)} \cdots \times D_{(n+1)s(n+1)}$ into subsets:
$
I_1 = \{s | s(n+1) = n+2\}; I_2 = \{s | s(n+1) = n+3\}; \cdots; I_{n+1} = \{s | s(n+1) = 2n+2\}
$. By induction, for each $j \in \{1, \cdots, n+1\}$, there exists an algebraic cycles $W_j$ such that
\[
\partial W_j = \sum_{s \in I_j} D_{1s(1)} \times D_{2s(2)} \cdots \times D_{(n+1)s(n+1)}.
\]
Therefore we have
\[
\partial \sum_{j = 1}^{n+1}  W_j = \sum_{j = 1}^{n+1} \partial W_j = \sum_{s} D_{1s(1)} \times D_{2s(2)} \cdots \times D_{(n+1)s(n+1)}.
\]
\end{proof}

\section{Linear correspondences and homotopy}
In this section, we want to generalize the idea of finding coboundaries for linear cycles as in the previous section.
\

Let us recall the definition of linear cycles. 
\begin{definition}
Consider the $2n$-th product of elliptic curves
$$
E^n\times E^n=E^{2n}=\{(x_1, \dots, x_{2n}) \mid x_i \in E\}.
$$
The subspace of $Z^n(E^{n} \times E^n)$ generated by
codimension $n$ subvarieties which can be expressed as
intersections of divisors of the following type:
\begin{align*}
&P_i=\{x_i=0\}\quad(i=1,\dots, 2n)
\\
&\Delta_{ij}=\{x_i=x_j\}\quad (1\leq i<j\leq 2n) 
\\
&\Delta^-_{ij}=\{x_i=-x_j\}\quad (1\leq i<j\leq 2n) 
\end{align*}
is  denoted by $LZ^n(E^{2n})=LZ^n(E^{2n},0)$. Its elements are called the linear cycles of $E$.
\end{definition}
Recall that $D_{ij} = P_i+P_j - \Delta_{ij}$.
\begin{definition}
Let $S\subset LZ^n(E^{2n},0)$ be the subspace of 
$LZ^n(E^{2n},0)$ generated by elements
\[
P_{i_1}\cdots P_{i_p} D_{j_1k_1} \cdots D_{j_{q}k_{q}}
\]
such that
\begin{enumerate}
\item 
$i_1, \cdots, i_p, j_1, \cdots, j_q, k_1, \cdots, k_q$ are distinct, $p+q=n$ and
\item{(Gelfand-Zetlin condition)}
\[
j_1 < j_2 < \cdots < j_q, k_1 < k_2 < \cdots < k_q, j_i < k_i(i = 1, \cdots, q).
\]
\end{enumerate}
\end{definition}
\begin{definition}
We define a subspace $T$ of $LZ^n(E^{2n}, 0)$ generated
(as for the meaning of generated, see below)
by the following elements.
\begin{enumerate}
\item
$P_i D_{jk} + D_{ij} D_{ik}$,
where $\{i,j,k\}$ is a subset of $\{1, \dots, 2n\}$ such that $\#I=3$.
\item
$D_{ij}D_{kl} + D_{ik} D_{jl}+D_{il}D_{jk}$,
where $\{i,j,k,l\}$ is a subset of $\{1, \dots, 2n\}$ such that $\#I=4$.
\end{enumerate}
\end{definition}

\begin{lemma}
We have a decomposition:
$LZ^{n}(E^{2n}, 0) = S \oplus T.$
\end{lemma}
\begin{proof}
From the proof of Lemma 4.3 of \cite{KT}, we have that $LZ^{n}(E^{2n}, 0) = S + T.$ 
\

Next we want to show that $S \cap T = \emptyset.$ Given any cycle $Z \in S \cap T$, we may assume that
\[
Z = \sum \alpha_{i_1, \cdots, i_p,j_1, 
\cdots, j_q,k_1, \cdots, k_q} P_{i_1}\cdots P_{i_p} D_{j_1k_1} \cdots D_{j_{q}k_{q}},
\]
where $\alpha_{i_1, \cdots, i_p,j_1, 
\cdots, j_q,k_1, \cdots, k_q}  \in \mathbb{Z}$, short for $\alpha_{i, j, k}$. Then we consider the image of $Z$ in $CH^n(E^{2n})$. Via Lemma 4.2 of \cite{KT}, we have $Z = 0 \in CH^n(E^{2n})$. Note that $P_{i_1} \cdots P_{i_p} D_{j_1k_1} \cdots D_{j_{q}k_{q}}$ are linearly independent in $CH^n(E^{2n})$. See Lemma 4.4, Corollary 4.5 of \cite{KT}. Hence $\alpha_{i,j,k} = 0$.
\end{proof}
We denote the projection of $LZ^n(E^{2n}, 0) $ to each direct summand $S$ 
(resp. $T$) by $\pi_S$(resp. $\pi_T$).

\begin{definition}
Let $I=\{i,j\}$ be a subset of $\{1, \dots, 2n\}$ such that $\#I=2$
and set
\begin{align*}
\varphi_I:Z^1(E^2,1)\times LZ^{n-1}(E^{2n-2},0)&\to
Z^{n}(E^{2n},1):
\\
(y_1,y_2)\times (x_1,\dots,x_{2n-2})&\mapsto
(x_1,\dots,\overset{i}{y_1},\dots,\overset{j}{y_2},\dots,x_{2n-2})
\end{align*}
Similarly, we let $I=\{i,j,k\}$ be a subset of $\{1, \dots, 2n\}$ such that $\#I=3$
and set
\begin{align*}
\varphi_I:Z^1(E^3,1)\times LZ^{n-1}(E^{2n-3},0)&\to
Z^{n}(E^{2n},1):
\\
(y_1,y_2,y_3)\times (x_1,\dots,x_{2n-3})&\mapsto
(x_1,\dots,\overset{i}{y_1},\dots,\overset{j}{y_2},\dots,\overset{k}{y_3}, \dots, x_{2n-3})
\end{align*}
Let $I=\{i,j,k,l\}$ be a subset of $\{1, \dots, 2n\}$ such that $\#I=4$
and set
\begin{align*}
\varphi_I: &Z^1(E^4,1)\times LZ^{n-1}(E^{2n-4},0)\to
Z^{n}(E^{2n},1):
\\
&(y_1,y_2,y_3,y_4)\times (x_1,\dots,x_{2n-4})\mapsto
\\
&(x_1,\dots,\overset{i}{y_1},\dots,\overset{j}{y_2},\dots,\overset{k}{y_3},\dots,\overset{l}{y_4}\dots,x_{2n-4})
\end{align*}
We define $LZ^n(E^{2n},1)$ as the subspace generated by elements of the form
$$
\varphi_I(Z(\theta)\times \gamma)
$$
where $Z(\theta)$ is an element in $Z^{1}(E^{2},1)$ (resp. $Z^{1}(E^{3},1), Z^{1}(E^{4},1)$)
defined by a combination of theta functions as in Theorem \ref{theta-theorem} and $\gamma$ is an element in
$LZ^{n-1}(E^{2n-2},0)$ (resp $LZ^{n-1}(E^{2n-3},0), LZ^{n-1}(E^{2n-4},0)$).
\end{definition}

\begin{proposition} \label{dec of linear cycles}
There is a map
$$
LZ^n(E^{2n},0)\xrightarrow{(\pi_S, h)} S\oplus LZ^n(E^{2n},1)
$$
such that
$$
\gamma-\pi_S(\gamma)=\partial(h(\gamma))
$$
\end{proposition}
\begin{proof}
Recall in the previous section (see (\ref{V123}), (\ref{V1234})), we have find algebraic cycles 
$\widetilde{W}_{Vijk}$ and $\widetilde{W}_{Vijkl}$, lying in 
$Z^2(E^{3}, 1)$ and $Z^2(E^{4}, 1)$, whose boundaries are $P_i D_{jk} + D_{ij} D_{ik}$ 
and $D_{ij}D_{kl} + D_{ik} D_{jl}+D_{il}D_{jk}$ respectively. Hence if we let $\widetilde{T}$ be the subspace of $Z^n(E^{2n}, 1)$ 
generated by the elements:
\begin{enumerate}
\item
$\widetilde{W}_{Vijk}$,
where $\{i,j,k\}$ is a subset of $\{1, \dots, 2n\}$ such that $\#I=3$.
\item
$\widetilde{W}_{Vijkl}$,
where $\{i,j,k,l\}$ is a subset of $\{1, \dots, 2n\}$ such that $\#I=4$, 
\end{enumerate}
then the boundary map $\partial: Z^n(E^{2n}, 1) \to Z^n(E^{2n}, 0)$ induces 
a surjective map $\widetilde{T} \xrightarrow{\partial} T$. 
We choose a section $\partial^{-1}$ of the surjecive map $\partial$.
Now we define a map
$LZ^n(E^{2n},0)\xrightarrow{(\pi_S, h)} S\oplus LZ^n(E^{2n},1)$ by
\begin{equation}
\begin{split}
LZ^n(E^{2n},0) \xrightarrow{(\pi_S, \pi_T)} S \oplus T \xrightarrow{(id, \partial^{-1})} S \oplus \widetilde{T} \to S\oplus LZ^n(E^{2n},1).
\end{split}
\end{equation}
It will satisfy $\gamma-\pi_S(\gamma)=\partial(h(\gamma))$ for $\gamma \in LZ^n(E^{2n},0)$.
\end{proof}
We consider the realization homomorphism
\begin{equation}
\begin{split}
Z^n(E^n\times E^n)
&\to 
\underline{\Hom}^0(\mathbf{R}\Gamma_M(E^n),\mathbf{R}\Gamma_M(E^n)) \\
&\to
\Hom(H^*_B(E^n, \mathbb{Q}),H^*_B(E^n, \mathbb{Q})),
\end{split}
\end{equation}

denoted it by $r$.
\begin{corollary}
Given any algebraic cycle $\gamma \in LZ^n(E^{2n})$ such that $r(\gamma) = 0$, then there exists an algebraic cycle $W_{\gamma} \in LZ^n(E^{2n}, 1)$ such that
$\partial W_{\gamma} = \gamma$.
\end{corollary}
\begin{proof}
It is enough to show that the kernel of the map $r: LZ^n(E^{2n}) \to \Hom(H^*_B(E^n),H^*_B(E^n))$ is contained in $T$. In fact, the map $r$ factors through the linear Chow group $CH^n_{\mathrm{lin}}(E^{2n})$, which is defined as the subspace of $CH^n(E^{2n})_{\mathbb{Q}}$ generated by the classes of abelian subvarieties in $E^{2n}$, equivalently $LZ^n(E^{2n}, 0)/T$ (see Section 4.1 in \cite{KT}). Note that we have:
\begin{equation}
\begin{split}
CH^n_{\mathrm{lin}}(E^{2n}) \otimes \mathbb{Q}_l \to &\Hom(H^*_B(E^n),H^*_B(E^n)) \otimes \mathbb{Q}_l \\
&\cong \Hom(H^*_{\acute{e}t}(E^n, \mathbb{Q}_l),H^*_{\acute{e}t}(E^n, \mathbb{Q}_l)),
\end{split}
\end{equation}
which is an injective map (see Corollary 5.3.3 in \cite{BL}). This implies the injectivity of $CH^n_{\mathrm{lin}}(E^{2n}) \to \Hom(H^*_B(E^n, \mathbb{Q}),H^*_B(E^n, \mathbb{Q}))$ and hence $T \subset \mathrm{Ker} (r)$.
\end{proof}

\section{Construction of $\pi_S$ and $h$}
In this section, we will give an explicit construction of $\pi_S$ and $h$.
First, we define subspaces 
\begin{equation}
\begin{split}
& S = LZ_{1,n} \subset LZ_{1,n-1}\subset \cdots\subset LZ_{1,1}= LZ_{2,n} \subset \\
&LZ_{2,n-1} \subset \cdots \subset LZ_{2,1} = LZ_2 \subset LZ_3=LZ^{n}(E^{2n},0)
\end{split}
\end{equation}
as follows.
\begin{enumerate}
 \item 
The subspace $LZ_{1,t}$ is generated by $S$ and elements of the form
\begin{equation}
\label{typical form} 
P_{i_1}\cdots P_{i_p} D_{j_1k_1} \cdots D_{j_{q}k_{q}}
\end{equation}
such that 
\begin{enumerate}
 \item 
$i_1, \cdots, i_p, j_1, \cdots, j_q, k_1, \cdots, k_q$ are distinct, $p+q=n$,
\item
$
j_1 < j_2 < \cdots < j_q, j_i < k_i(i = 1, \cdots, q), 
$ and
\item
$\min\{t'\mid k_{t'}>k_{t'+1}\}\geq t
$.
\end{enumerate}
Then we have $S = LZ_{1,n} \subset LZ_{1,n-1}\subset\cdots\subset LZ_{1,1}$. 
 \item 
The subspace $LZ_{2} = LZ_{2,1}$ is generated by elements of the form
\eqref{typical form} 
such that 
\[\{i_1, \cdots, i_p\}\cap \{j_1, \cdots, j_q, k_1, \cdots, k_q\}=\emptyset\]
and $p+q=n$. In general, the subspace $LZ_{2,s}$ is generated by elements of the form
\eqref{typical form} 
such that 
\begin{enumerate}
\item
$\{i_1, \cdots, i_p\}\cap \{j_1, \cdots, j_q, k_1, \cdots, k_q\}=\emptyset$
and $p+q=n$,
\item
$i_1, \cdots, i_p$ are distinct, $j_i < k_i(i = 1, \cdots, q)$,
\item
$
j_1 \leq j_2 \leq \cdots \leq j_q
$ and $\min\{s'\mid j_{s'} = j_{s'+1}\}\geq s
$.
\end{enumerate}
 \item 
The subspace $LZ_{3}$ is generated by elements of the form
$$
P_{i_1}\cdots P_{i_p} D^{\epsilon_1}_{j_1k_1} \cdots 
D_{j_{q}k_{q}}^{\epsilon_q}
$$
such that 
$\epsilon_i=\pm 1$\footnote{Here we use the notation $D^{-1}_{jk} = D^{-}_{jk}$.},
$\{i_1, \cdots, i_p\}\cap \{j_1, \cdots, j_q, k_1, \cdots, k_q\}=\emptyset$
and $p+q=n$.
\end{enumerate}
We construct homomorphisms
\begin{align*}
&LZ_3=LZ^n(E^{2n},0)\xrightarrow{(\pi_3, h_3)} LZ_2\oplus LZ^n(E^{2n},1) 
\\
&LZ_{2,s} \xrightarrow{(\pi_{2,s}, h_{2,s})} LZ_{2,s+1}\oplus LZ^n(E^{2n},1) \quad(s=1,\dots,n-1)
\\
&LZ_{1,t}\xrightarrow{(\pi_{1,t}, h_{1,t})} LZ_{1,t+1}\oplus LZ^n(E^{2n},1) 
\quad (t=1,\dots, n-2)\\
&LZ_{1,n-1}\xrightarrow{(\pi_{1,n-1}, h_{1,n-1})} S\oplus LZ^n(E^{2n},1) 
\end{align*}
such that
\begin{align*}
\pi_{\star}(\gamma)=\gamma+\partial h_{\star}(\gamma) 
\quad(\star=(1,1),\dots,(1,n-1),(2,1),\dots,(2,n-1),3)
\end{align*}
We set 
$$
\Pi_{1,n-1}=\pi_{1,n-1}\circ\pi_{1,n-2}\circ\cdots \circ\pi_{1,1}\circ \pi_{2,n-1} \circ \dots \circ \pi_{2,1}\circ\pi_3
:LZ^n(E^{2n},0) \to S,
$$
$$
\Pi_{1,t-1}=\pi_{1,t-1}\circ\cdots \circ\pi_{1,1}\circ \pi_{2,n-1} \circ \dots \circ \pi_{2,1} \circ\pi_3
:LZ^n(E^{2n},0) \to LZ_{1,t},
$$
$$
\Pi_{2}= \pi_{2,n-1} \circ \dots \circ \pi_{2,1} \circ\pi_3
:LZ^n(E^{2n},0) \to LZ_{1,1},
$$
$$
\Pi_{2,s-1}= \pi_{2,s-1} \circ \dots \circ \pi_{2,1} \circ\pi_3
:LZ^n(E^{2n},0) \to LZ_{2,s}.
$$
Note that $\Pi_2 = \pi_{2,n-1}$.
Then we have
\begin{equation} \nonumber
\begin{split}
\Pi_{1,n-1}
=&\Pi_{1,n-2}
+\partial h_{1,n-1}\circ\Pi_{1,n-2}
\\
=&\Pi_{1,n-3}
+\partial h_{1,n-2}\circ\Pi_{1,n-3}
+\partial h_{1,n-1}\circ\Pi_{1,n-2}
\\
&\cdots
\\
=&\Pi_{1,1}
+\partial h_{1,2}\circ\Pi_{1,1}+\cdots
+\partial h_{1,n-1}\circ\Pi_{1,n-2}
\\
=&\Pi_{2}
+\partial h_{1,1}\circ\Pi_{2}
+\sum_{i=1}^{n-2}\partial h_{1,i+1}\circ\Pi_{1,i}
\\
=& \Pi_{2,n-2} + \partial h_{2,n-1} \circ \Pi_{2,n-2} + \partial h_{1,1} \Pi_2 + \sum_{i=1}^{n-2}\partial h_{1,i+1}\circ\Pi_{1,i}
\\
& \cdots
\\
=&\pi_{3}
+\partial h_{2,1}\circ\pi_{3}
+\sum_{i=1}^{n-2}\partial h_{2,i+1}\circ\Pi_{2,i} + \partial h_{1,1} \Pi_2
+\sum_{i=1}^{n-2}\partial h_{1,i+1}\circ\Pi_{1,i}
\\
=& \mathrm{id}
+\partial h_{3}+
\partial h_{2,1}\circ\pi_{3}
+\sum_{i=1}^{n-2}\partial h_{2,i+1}\circ\Pi_{2,i}
\\
&+\partial h_{1,1}\circ\Pi_{2}
+\sum_{i=1}^{n-2}\partial h_{1,i+1}\circ\Pi_{1,i}
\end{split}
\end{equation}
Therefore by setting, 
\[
h=h_{3}
+ h_{2,1}\circ\pi_{3}
+\sum_{i=1}^{n-2} h_{2,i+1}\circ\Pi_{2,i}
+ h_{1,1}\circ\Pi_{2}
+\sum_{i=1}^{n-2} h_{1,i+1}\circ\Pi_{1,i},
\]
we have
$\Pi_{1,n-1}(\gamma)=\gamma+\partial(h(\gamma)).$
We recall that:
\begin{enumerate}
\item
$\partial \widetilde{W}_{Vij} = D_{ij} + (id \times \sigma)_*D_{ij} = D_{ij} + D_{ij}^-$;
\item
$\partial \widetilde{W}_{Vijk} = P_iD_{jk} + D_{ij}D_{ik}$;
\item
$\partial \widetilde{W}_{Vijkl} = D_{ij}D_{kl}+D_{ik}D_{jl}+D_{il}D_{jk}$.
\end{enumerate}
\textbf{The construction $(\pi_3,h_3)$}:
\

We define the morphism $LZ_3=LZ^n(E^{2n},0)\xrightarrow{(\pi_3, h_3)} LZ_2\oplus LZ^n(E^{2n},1) $ as follows:
\begin{eqnarray}\nonumber
&\pi_3(P_{i_1}\cdots P_{i_p} D^{\varepsilon_1}_{j_1k_1} \cdots 
D_{j_{q}k_{q}}^{\varepsilon_q})
= \varepsilon_1 \cdots \varepsilon_q P_{i_1}\cdots P_{i_p} D_{j_1k_1} \cdots 
D_{j_{q}k_{q}}; \\ \nonumber
&h_3(P_{i_1}\cdots P_{i_p} D^{\varepsilon_1}_{j_1k_1} \cdots 
D_{j_{q}k_{q}}^{\varepsilon_q})  \nonumber
=  \varepsilon_2 \cdots \varepsilon_q P_{i_1}\cdots P_{i_p} (\frac{1-\varepsilon_1}{2} \widetilde{W}_{Vj_1k_1}) D_{j_2k_2} \cdots D_{j_qk_q}
 \\ \nonumber
&+ \varepsilon_3 \cdots \varepsilon_q P_{i_1}\cdots P_{i_p} D^{\varepsilon_1}_{j_1k_1} (\frac{1-\varepsilon_2}{2} \widetilde{W}_{Vj_2k_2}) D_{j_3k_3} \cdots D_{j_qk_q}  \nonumber
+ \cdots \\ \nonumber
&+ P_{i_1}\cdots P_{i_p} D^{\varepsilon_1}_{j_1k_1} \cdots D^{\varepsilon_{q-1}}_{j_{q-1}k_{q-1}} (\frac{1-\varepsilon_q}{2} \widetilde{W}_{Vj_qk_q}).
\end{eqnarray}

\noindent \textbf{The construction of $(\pi_{2,s}, h_{2,s})$}:
\

Without loss of generality, we may further assume that the generators 
\[P_{i_1}\cdots P_{i_p} D_{j_1k_1} \cdots D_{j_{q}k_{q}} \in LZ_{2,s} \]
satisfy: $k_s < k_{s+1}$ if $j_s = j_{s+1}$.
\

We define the morphism $LZ_{2,s} \xrightarrow{(\pi_{2,s}, h_{2,s})} LZ_{2,s+1}\oplus LZ^n(E^{2n},1)$ as follows:
\begin{itemize}
\item
$P_{i_1}\cdots P_{i_p} D_{j_1k_1} \cdots D_{j_{q}k_{q}}$ maps to itself, if $\min\{s'\mid j_{s'} = j_{s'+1}\} \geq s+1$;
\item
If $\min\{s'\mid j_{s'} = j_{s'+1}\} = s$, we define:
\begin{equation} \nonumber
\begin{split}
&\pi_{2,s}(P_{i_1}\cdots P_{i_p} D_{j_1k_1} \cdots D_{j_{q}k_{q}}) = - P_{i_1}\cdots P_{i_p} \cdots P_{j_s} \\
&D_{j_1k_1} \cdots D_{j_{s-1}k_{s-1}} D_{k_s k_{s+1}} D_{j_{s+2} k_{s+2}} \cdots D_{j_qk_q}.
\end{split}
\end{equation}
Here we may re-arrange the orders of elements. Note that $k_s > j_s$, and hence the image of $\pi_{2,s}$ lies in $LZ_{2,s+1}$. Then we define:
\begin{equation} \nonumber
\begin{split}
&h_{2,s}(P_{i_1}\cdots P_{i_p} D_{j_1k_1} \cdots D_{j_{q}k_{q}}) = - P_{i_1}\cdots P_{i_p} D_{j_1k_1} \\
&\cdots D_{j_{s-1}k_{s-1}} \widetilde{W}_{Vj_sk_sk_{s+1}} D_{j_{s+2} k_{s+2}} \cdots D_{j_qk_q}.
\end{split}
\end{equation}
\end{itemize}
\noindent \textbf{The construction of $(\pi_{1,t}, h_{1,t})$}:
\

We define the morphism $LZ_{1,t}\xrightarrow{(\pi_{1,t}, h_{1,t})} LZ_{1,t+1}\oplus LZ^n(E^{2n},1)$ as follows:
\begin{itemize}
\item
$P_{i_1}\cdots P_{i_p} D_{j_1k_1} \cdots D_{j_{q}k_{q}}$ maps to itself, if $\min\{t'\mid k_{t'}>k_{t'+1}\}\geq t+1
$;
\item
If $\min\{t'\mid k_{t'}>k_{t'+1}\} =  t$, we define:
\begin{equation} \nonumber
\begin{split}
&\pi_{1,t}(P_{i_1}\cdots P_{i_p} D_{j_1k_1} \cdots D_{j_{q}k_{q}}) = - P_{i_1}\cdots P_{i_p} D_{j_1k_1} \cdots D_{j_{t-1}k_{t-1}} \\
&D_{j_tj_{t+1}}D_{k_{t+1}k_t} D_{j_{t+2}k_{t+2}}\cdots D_{j_{q}k_{q}} - P_{i_1}\cdots P_{i_p} D_{j_1k_1} \cdots D_{j_{t-1}k_{t-1}} \\
& D_{j_tk_{t+1}} D_{\min\{k_t, j_{t+1}\} \max\{k_t, j_{t+1}\}}D_{j_{t+2}k_{t+2}}\cdots D_{j_{q}k_{q}} \\
& \mathrm{and} \\
&h_{1,t}(P_{i_1}\cdots P_{i_p} D_{j_1k_1} \cdots D_{j_{q}k_{q}}) = -P_{i_1}\cdots P_{i_p} D_{j_1k_1} 
\cdots D_{j_{t-1}k_{t-1}} \\
&\widetilde{W}_{Vj_tk_tj_{t+1}k_{t+1}}D_{j_{t+2}k_{t+2}}\cdots D_{j_{q}k_{q}}. 
\end{split}
\end{equation}
\noindent Note that we use here: $j_{t+1} < k_{t+1} < k_t$ and $k_{t+1} < \max\{k_t, j_{t+1}\}$.
\end{itemize}

\section{The resolution of DG complexes}
In this section, we want to study the resolution of $\mathbf{R} \Gamma_M(E^n)$ for any positive integer $n$.
\

We recall some properties of group rings of symmetric groups. For each positive integer $n$, we denote the symmetric group of a totally ordered set $A$ with $n$ elements  by $\Sigma(A)$. It is known that the representation theory of $\Sigma(A)$ is determined by the Young tableaux. If a labeling of a tableau is strictly increasing to the right across the rows and strictly increasing down the columns, then it is called a standard tableau and the set of standard tableaux of $A$ is denoted by $Tab(A)$.  For a tableau $Y$ consisting of $A$, we can define a projector $e_Y \in \mathbb{Q}[\Sigma(A)]$, which are called a Young symmetrizer. All the symmetrizers indexed by standard Young tableaux lead a decomposition of $\mathbb{Q}[\Sigma_n]$, i.e, we have
\begin{equation} \label{dec of sym}
\mathbb{Q}[\Sigma_n] = \bigoplus \mathbb{Q}[\Sigma_n]e_Y. 
\end{equation}
See Section 4.2 in \cite{FH}.
\

Given an elliptic curve $E$, we let $E_i$ be a copy of $E$ indexed by $i \in A$. The self product $\prod_{i \in A} E$ is denoted by $E^A$. Given a Young tableaux $Y$ of $A$, we let:
\[i_Y^* \in \underline{\Hom}^0(\mathbf{R} \Gamma_M(E^A), (\mathbb{Q} \oplus V[-1] \oplus \mathbb{Q}(-1)[-2])^{\otimes A})\]
be the composition of homomorphisms
\begin{equation}
\begin{split}
& \mathbf{R} \Gamma_M(E^A) = \mathbf{R} \Gamma_M(E)^{\otimes A} \xrightarrow{\alpha^{\otimes A}} (\mathbb{Q} \oplus V[-1] \oplus \mathbb{Q}(-1)[-2])^{\otimes A} \\
& \xrightarrow{e_Y} (\mathbb{Q} \oplus V[-1] \oplus \mathbb{Q}(-1)[-2])^{\otimes A}
\end{split}
\end{equation}
and
\[\pi_Y^* \in \underline{\Hom}^0((\mathbb{Q} \oplus V[-1] \oplus \mathbb{Q}(-1)[-2])^{\otimes A}, \mathbf{R} \Gamma_M(E^A))\]
be the composition of homomorphisms
\begin{equation}
\begin{split}
& (\mathbb{Q} \oplus V[-1] \oplus \mathbb{Q}(-1)[-2])^{\otimes A}  \xrightarrow{e_Y} (\mathbb{Q} \oplus V[-1] \oplus \mathbb{Q}(-1)[-2])^{\otimes A} \\
& \xrightarrow{\beta^{\otimes A}} \mathbf{R} \Gamma_M(E)^{\otimes A} = \mathbf{R} \Gamma_M(E^A) .
\end{split}
\end{equation}
\begin{lemma}
There exists an algebraic cycle 
\[H_A  \in \underline{\Hom}^{-1}(\mathbf{R} \Gamma_M(E^A), \mathbf{R} \Gamma_M(E^A))\] such that
\[
\partial(H_A) = id - \sum_{Y \in Tab(A)} \pi_Y^* i_Y^*,
\]
which is called the standard homotopy for id-const.
\end{lemma}
\begin{proof}
According to the equation (\ref{dec of sym}), we know that
\begin{equation}
\sum_{Y \in Tab(A)} \pi_Y^* i_Y^* = \beta^{\otimes A} (\sum_{Y \in Tab(A)} e_Y^2 ) \alpha^{\otimes A} = \beta^{\otimes A} \alpha^{\otimes A} = (\beta \alpha)^{\otimes A}.
\end{equation}
Then using Proposition 2.3, we will get the desired $H_A$, i.e. the $A$-th product of $-W_{\mathbf{R} \Gamma_M(E)}$.
\end{proof}
\begin{remark}
We recall that $\Gamma_A = (\mathbb{Z}/2\mathbb{Z})^A \rtimes \Sigma(A)$ acts on $E^A$, where $\Sigma(A)$ permutes the components of $E^A$ and the $i$-th generator $(0, \cdots, 1, \cdots, 0)$ in $ (\mathbb{Z}/2\mathbb{Z})^A$ acts on the $i$-th component $E$ of $E^A$ by the inversion $x \to -x$. The sign character is defined on $(\mathbb{Z}/2\mathbb{Z})^A$, which is denoted by $\rho_A$ and it can be extended to $\Gamma_A$ trivially. The invariant part of $\mathbf{R} \Gamma_M(E^A)$ under the action of $\rho_A$ is just $V^A$. Because the intersection product between $H_A$ and the graph of $\rho_A$ is commutative, the restriction of $H_A$ on $V$ leads an algebraic cycle $H_{V^{\otimes A}}$ in $\underline{\Hom}^{-1}(V^{\otimes A}, V^{\otimes A})$ satisfies its boundary is $id_{V^{\otimes A}} - \sum_{Y \in Tab(A)} \pi_Y^* i_Y^*$.
\end{remark}

\begin{definition}
A DG complex is called a Schur complex if its terms are a direct sum of $e_Y(V^{\otimes n})(m)[p]$, which is isomorphic to $S_{\lambda}(V^{\otimes n})(m)[p]$, where $\lambda$ is the Young diagram underlying the Young tableau $Y$ and $S_{\lambda}$ is the Schur functor associated to $\lambda$.  
\end{definition}
Now we want to construct a resolution of a DG complex
\begin{equation} \label{DG comples}
\begin{split}
& \cdots  \to \oplus \mathbf{R} \Gamma_M(E^{n_1})(m_1)[p_1] \\
& \xrightarrow{f_{21}} \oplus \mathbf{R} \Gamma_M(E^{n_2})(m_2)[p_2] \\
& \xrightarrow{f_{32}} \oplus \mathbf{R}\Gamma_M(E^{n_3})(m_3)[p_3] \to \cdots 
\end{split}
\end{equation}
via the Schur complexes. By the definition of DG complexes, we have $\partial(f_{21}) = \partial(f_{32}) = 0$ and $\partial(f_{31}) = f_{32}f_{21}$, where $f_{31}:  \oplus \mathbf{R} \Gamma_M(E^{n_1})(m_1)[p_1] \to \oplus \mathbf{R}\Gamma_M(E^{n_3})(m_3)[p_3]$.
\

More precisely, we can form another DG complex:
\begin{equation} \label{resolution}
\begin{split}
& \cdots  \to \oplus (\mathbb{Q} \oplus V[-1] \oplus \mathbb{Q}(-1)[-2])^{\otimes n_1}(m_1)[p_1]  \\ &\xrightarrow{F_{21}} \oplus (\mathbb{Q} \oplus V[-1] \oplus \mathbb{Q}(-1)[-2])^{\otimes n_2}(m_2)[p_2] \\
& \xrightarrow{F_{32}} \oplus (\mathbb{Q} \oplus V[-1] \oplus \mathbb{Q}(-1)[-2])^{\otimes n_3}(m_3)[p_3] \to \cdots 
\end{split}
\end{equation}
by letting
\[
F_{21} =( \sum_{Y \in Tab(n_2)} i_Y^*) f_{21}  (\sum_{Y \in Tab(n_1)}  \pi_Y^*), 
\]
\[
F_{32} =  (\sum_{Y \in Tab(n_3)} i_Y^*) f_{32} (\sum_{Y \in Tab(n_2)}  \pi_Y^*)
\]
and
\[
F_{31} = (\sum_{Y \in Tab(n_3)} i_Y^*) (f_{31} - f_{32} H_{n_2} f_{21} )(\sum_{Y \in Tab(n_1)}  \pi_Y^*).
\]
Then we have:
\begin{equation} \nonumber
\begin{split}
\partial(F_{31}) =& (\sum_{Y \in Tab(n_3)} i_Y^*) (\partial(f_{31}) - f_{32} \partial (H_{n_2}) f_{21}) (\sum_{Y \in Tab(n_1)}  \pi_Y^*) \\
= & (\sum_{Y \in Tab(n_3)} i_Y^*)(f_{32} f_{21} - f_{32}(id - \sum_{Y \in Tab(n_2)}  \pi^*_Y i^*_Y)f_{21})(\sum_{Y \in Tab(n_1)}  \pi_Y^*) \\
=& (\sum_{Y \in Tab(n_3)} i_Y^*) f_{32} (\sum_{Y \in Tab(n_2)}  \pi^*_Y i^*_Y) f_{21} (\sum_{Y \in Tab(n_1)}  \pi_Y^*) \\
=&(\sum_{Y \in Tab(n_3)} i_Y^*) f_{32} (\sum_{Y \in Tab(n_2)}  \pi^*_Y) (\sum_{Y \in Tab(n_2)}  i^*_Y) f_{21} (\sum_{Y \in Tab(n_1)}  \pi_Y^*) \\
=& F_{32} F_{21}.
\end{split}
\end{equation}
We let $h_{21} = ( \sum_{Y \in Tab(n_2)} i_Y^*) f_{21} H_{n_1}$, which is a morphism:
\[
\oplus \mathbf{R} \Gamma_M(E^{n_1})(m_1)[p_1] \to \oplus (\mathbb{Q} \oplus V[-1] \oplus \mathbb{Q}(-1)[-2])^{\otimes n_2}(m_2)[p_2] 
\]
$h_{32} =  ( \sum_{Y \in Tab(n_3)} i_Y^*) f_{32} H_{n_2}$, which is a morphism:
\[
\oplus \mathbf{R} \Gamma_M(E^{n_2})(m_2)[p_2] \to \oplus (\mathbb{Q} \oplus V[-1] \oplus \mathbb{Q}(-1)[-2])^{\otimes n_3}(m_3)[p_3] 
\]
and similarly $h_{31} = ( \sum_{Y \in Tab(n_3)} i_Y^*) (f_{31} - f_{32}H_{n_2}f_{21}) H_{n_1}$. Then we have:
\begin{equation} \nonumber
\partial(h_{21}) = ( \sum_{Y \in Tab(n_2)} i_Y^*) f_{21} - F_{21}( \sum_{Y \in Tab(n_1)} i_Y^*),
\end{equation}
\begin{equation} \nonumber
\partial(h_{32}) = ( \sum_{Y \in Tab(n_3)} i_Y^*) f_{32} - F_{32}( \sum_{Y \in Tab(n_2)} i_Y^*),
\end{equation}
and
\begin{equation} \nonumber
\partial(h_{32}) = [( \sum_{Y \in Tab(n_3)} i_Y^*) f_{31} - h_{32}f_{21}] - [F_{31}( \sum_{Y \in Tab(n_1)} i_Y^*) -F_{32}h_{21}].
\end{equation}
Moreover, the general formulas are given by
\begin{equation}
\begin{split}
F_{p, q} &=   (\sum_{Y \in Tab(n_p)} i_Y^*) f_{p, q}  (\sum_{Y \in Tab(n_q)} \pi_Y^*) \\
&- \sum_{q < r < p}  (\sum_{Y \in Tab(n_p)} i_Y^*) f_{p,r} H_{n_r} f_{r, q} (\sum_{Y \in Tab(n_q)} \pi_Y^*) \\
&+ \sum_{q< r_1 < r_2 < q}  (\sum_{Y \in Tab(n_p)} i_Y^*) f_{p,r_1} H_{n_{r_1}} f_{r_1, r_2} H_{n_{r_2}} f_{r_2, q} (\sum_{Y \in Tab(n_q)} \pi_Y^*) \\
& - \cdots,
\end{split}
\end{equation}
and 
\begin{equation}
\begin{split}
h_{p,q} &=   (\sum_{Y \in Tab(n_p)} i_Y^*) f_{p,q} H_{n_q} \\
&-\sum_{q<r<p}  (\sum_{Y \in Tab(n_p)} i_Y^*) f_{p,r} H_{n_r} f_{r, q} H_{n_q} \\
&+\sum_{q< r_1 < r_2 < q}  (\sum_{Y \in Tab(n_p)} i_Y^*)f_{p,r_1} H_{n_{r_1}} f_{r_1, r_2} H_{n_{r_2}} f_{r_2, q} H_{n_q} \\
&- \cdots.
\end{split}
\end{equation}
This implies that, for a given DG complex $K$ of the type (\ref{DG comples}), there is an object $K^{'}$ of the type (\ref{resolution}) together with a quasi-isomorphism $K \to K^{'}$.

\section{The motivic fundamental groups of punctured elliptic curves}
In this section, we recall the constructions of the nilpotent completion of the fundamental group of an punctured elliptic curve following the work of Deligne and Goncharov. See \cite[Section 3]{DG} and \cite{T10}. Eisenreich \cite{Eis} called these constructions the motivic fundamental group. In particular, we follow the notions used in \cite{Eis}.
\

Let $X$ be a smooth $S$-scheme with smooth subschemes $D_1, D_2, \cdots, D_n$ in $X$. For a given index set $I = (1 \leq i_1 < \cdots < i_s \leq n)$, we denote the intersection of all these $D_{j}$ with $j \in I$ by $D_I$. For simplicity, we use $\mathbb{Q}_X$ to denote the motive of $X$ in the category of smooth motives over $S$. For example, if $X = E^n$, we have $\mathbb{Q}_X = \mathbf{R} \Gamma_M(E^n)$ as we defined before.
\begin{definition}
The relative motive $\mathbb{Q}_{(X; D_1, \cdots, D_n)}$ is defined as the complex:
\[
\mathbb{Q}_X \to \bigoplus^n_{i = 1} \mathbb{Q}_{D_i} \to \cdots \to \bigoplus^n_{|I| = s} \mathbb{Q}_{D_I} \to \cdots \to \mathbb{Q}_{D_1 \cap \cdots \cap D_n}
\]
sitting in degree $0$ up to $n$, where the differential in degree $s$ is given by the alternating sum of
\[
\partial^s = \sum_{|I| = s} \sum^n_{i = 1} (-1)^i \partial^s_{I, i}
\]
and each component is defined by
\[
\partial^s_{I, i} = \begin{cases}
i^*_{D_I \subset D_{I \cup \{i\}}} & \mathbf{if} \ i \notin I, \\
0  & \mathbf{if} \ i \in I.
\end{cases} 
\]
\end{definition}
Our main example is the following:
\begin{example}
Let $X$ be a smooth, quasi-projective scheme over a reduced scheme $S$ and let $x, y: S \to X$ be two sections of the structure map $\pi: X \to S$. We are interested in the following subsets of $X^n$:
\begin{eqnarray} \nonumber
D^{(n)}_0 & =& x(S) \times X^{n-1} \\ \nonumber
D^{(n)}_i &=& \{ x_i = x_{i+1} \} \subset X^n \ \mathbf{for} \ 1\leq i \leq n-1 \\ \nonumber
D^{(n)}_n &=& X^{n-1} \times y(S).
\end{eqnarray}
Then we have a natural isomorphism
\begin{equation} \label{bar}
b_{\leq 0}\left(\mathbb{Q}_{(X; D_0^{(n)}, \cdots, D_n^{(n)})}[n]\right) \cong \widetilde{B}^{mot}_n(X \mid S)_{x, y}, 
\end{equation}
where $\widetilde{B}^{mot}_n(X \mid S)_{x, y}$ is the $n$-th normalized motivic bar complex of $X$, which is defined in \cite[Definition 2.4.6]{Eis} and $b_{\leq 0}$ is the brutal truncation from above after degree $0$. The proof of (\ref{bar}) can be found in \cite[Theorem 2.7.3]{Eis} or \cite[Section 3]{DG}.
\end{example}
Now we move to the concrete examples and let $S$ be a field $k$.
\subsection{The case of $\mathbb{P}^{1} - \{0, 1, \infty\}$}
Consider $X = \mathbb{P}^{1} - \{0, 1, \infty\}$, which is defined over a field $k$. Take two distinct rational points $a, b \in X$. We apply the above construction and let $n = 2$ for simplicity. Then we get:
\begin{equation} \label{fundamental gp for P1minus3}
\mathbb{Q}_{X^2} \to \mathbb{Q}_{a \times X} \oplus \mathbb{Q}_{\Delta X} \oplus \mathbb{Q}_{X \times b} \rightarrow \mathbb{Q}_{(a, a)} \oplus \mathbb{Q}_{(a,b)} \oplus \mathbb{Q}_{(b,b)},
\end{equation}
where $\Delta X \subset X^2$ is the diagonal. It is know that, via the Mayer-Vietoris property, we have:
\[
(\mathbb{Q}_{\mathbb{A}^1} \to \mathbb{Q}_{\mathbb{A}^1-\{0\}} \oplus  \mathbb{Q}_{\mathbb{A}^1-\{1\}})\xrightarrow{q.i.} \mathbb{Q}_{X}.
\]
Next we can apply the construction of Section 4 if we consider the homotopy
\[
H_n \in \underline{\Hom}^{-1}(\mathbb{Q}_{\mathbb{A}^n}(q)[p], \mathbb{Q}_{\mathbb{A}^n}(q)[p])
\]
defined by the cycle
\[
\{(x, x-tx, t) \in \mathbb{A}^n \times \mathbb{A}^n \times \mathbb{A}^1 \mid x \in \mathbb{A}^n, t \in \mathbb{A}^1\} \in Z^n(\mathbb{A}^n \times \mathbb{A}^n, 1).
\]
Then one may consider its resolution and compute the higher homotopy. For example, we consider the component
\[
\mathbb{Q}_{\mathbb{A}^2} \to \mathbb{Q}_{(\mathbb{A}^1-\{0\}) \times \mathbb{A}^1} \to \mathbb{Q}_{(\mathbb{A}^1-\{0\}) \times (\mathbb{A}^1-\{1\})} \to \mathbb{Q}_{(\mathbb{A}-\{0\}) \times b} \to \mathbb{Q}_{(b,b)}.
\]
or a simplified component
\[\xymatrixcolsep{1pc}\xymatrix{
 \mathbb{Q}_{(\mathbb{A}^1-\{0\} )\times (\mathbb{A}^1-\{1\} )} \ar[r] \ar[d]  &\mathbb{Q}_{(\mathbb{A}^1-\{0\} ) \times b}\ar[r] \ar[d]  &\mathbb{Q}_{(b,b)} \ar[d] \\
 \mathbb{Q}(-2)[-2] \myarat{-->}[r] \myarat<1ex>[u] & (\mathbb{Q} \oplus \mathbb{Q}(-1)[-1])  \myarat{-->}[r] \myarat<1ex>[u] & \mathbb{Q} \myarat<1ex>[u]\\
}
\]
The formula in the end of Section 4 provides us a cycle
\[
(t, 1-t, 1 - \frac{b}{t}) \in Z^2(\mathrm{Spec}(k), 3).
\]




\subsection{The punctured elliptic curve}
Now we go back to the motivic fundamental group of a punctured elliptic curve $X = E - \{O\}$, where $E$ is an elliptic curve defined over $\mathrm{Spec}(k)$. Notice that, if we choose $x$ and $y$ as a pair of different points $a, b \in X$,  the Betti realization of the relative motive $\mathbb{Q}_{(X^n; D_0^{(n)}, \cdots, D_n^{(n)})}$:
\begin{equation}
\mathbb{Q}_{X^n} \to \bigoplus_{i=0}^n \mathbb{Q}_{D^{(n)}_i} \to \cdots \to \bigoplus_{i=0}^n \mathbb{Q}_{ \cap_{j \neq i} D^{(n)}_j}
\end{equation}
is the dual of the $n$-th truncation of the nilpotent completion of the fundamental group of a punctured elliptic curve $\mathbb{Q}[\pi_1(X, a, b)]/I^n$. See Beilinson's construction in \cite[Section 3]{DG}.
\

We fix isomorphisms:
\begin{eqnarray}  \label{fix iso}
(p_2, \cdots, p_n) &:& D^{(n)}_0 \to X^{n-1} \\ \nonumber
(p_1, p_2, \cdots, \widehat{p_{i+1}}, \cdots, p_n) &:& D^{(n)}_i \to X^{n-1}  \ \mathbf{if} \ 1 \leq i \leq n-1 \\ \nonumber
(p_1, p_2, \cdots, p_{n-1}) &:& D^{(n)}_n \to X^{n-1},
\end{eqnarray}
where $p_i: X^n \to X$ is the $i$-th projection. Under these isomorphisms together with
\[
(\mathbb{Q}_{0}(-1)[-2] \to \mathbb{Q}_E)  \xrightarrow{\cong} \mathbb{Q}_X
\]
the complex of the relative motive $\mathbb{Q}_{(X^n; D_0^{(n)}, \cdots, D_n^{(n)})}$ can be viewed as a DG complex with each term having the type as we have seen in the previous section. For example:
\begin{example} \label{gen of Totaro}
Assume that $n = 2$. The relative motive $\mathbb{Q}_{(X^2; D^{(2)}_0, D^{(2)}_1, D^{(2)}_2)}$:
\[
\mathbb{Q}_{X^2} \xrightarrow{\gamma_1} \mathbb{Q}_{a \times X} \oplus \mathbb{Q}_{\Delta X} \oplus \mathbb{Q}_{X \times b} \xrightarrow{\gamma_2} \mathbb{Q}_{(a,a)} \oplus \mathbb{Q}_{(a,b)} \oplus \mathbb{Q}_{(b,b)},
\] 
can be expressed as
\begin{scriptsize}
\begin{equation} \nonumber
\xymatrix{
\mathbb{Q}_{(0,0)}(-2)[-4] \ar[r] &\mathbb{Q}_{E\times \{0\}}(-1)[-2] \oplus \mathbb{Q}_{\{0\} \times E}(-1)[-2] \ar[r]\ar[d] &\mathbb{Q}_{\mathbb{E}^2} \ar[d] \\
& \mathbb{Q}_{(a,0)}(-1)[-2] \oplus \mathbb{Q}_{(0,0)}(-1)[-2] \oplus \mathbb{Q}_{(0,b)}(-1)[-2]  \ar[r] & \mathbb{Q}_{a \times E} \oplus \mathbb{Q}_{\Delta E} \oplus \mathbb{Q}_{E \times b} \ar[d]\\
& &\mathbb{Q}_{(a,a)} \oplus \mathbb{Q}_{(a,b)} \oplus \mathbb{Q}_{(b,b)}
}
\end{equation}
\end{scriptsize}
A priori, the above diagram is a DG complex. Since the morphisms arise from the geometric ones, the double complex is a strict complex. Then we can find resolutions of the above DG complexes via replacing each term by its resolution.
Let us consider the component
\[
\mathbb{Q}_{(0,0)}(-2)[-4] \to \mathbb{Q}_{\mathbb{E} \times 0}(-1)[-2] \to \mathbb{Q}_{\mathbb{E}^2} \to \mathbb{Q}_{a \times E} \to \mathbb{Q}_{(a,a)}.
\]
as the previous case. Then it gives rises to the bottom DG-complex:
\begin{scriptsize}
\[\xymatrixcolsep{1pc}\xymatrix{
\mathbb{Q}_{(0,0)}(-2)[-4] \ar[r] \ar[d] & \mathbb{Q}_{\mathbb{E} \times 0}(-1)[-2] \ar[r] \ar[d]  & \mathbb{Q}_{\mathbb{E}^2} \ar[r] \ar[d]  &\mathbb{Q}_{a \times E}\ar[r] \ar[d]  &\mathbb{Q}_{(a,a)} \ar[d] \\
\mathbb{Q}(-2)[-4]  \myarat{-->}[r] \myarat<1ex>[u] & (\mathbb{Q}(-1)[-2]\oplus V(-1)[-3] \oplus \mathbb{Q}(-2)[-4]) \myarat{-->}[r] \myarat<1ex>[u] & \mathbb{Q}_{\mathbb{E}^2} \myarat{-->}[r] \myarat<1ex>[u] & (\mathbb{Q}\oplus V[-1] \oplus \mathbb{Q}(-1)[-2])  \myarat{-->}[r] \myarat<1ex>[u] & \mathbb{Q} \myarat<1ex>[u]\\
}
\]
\end{scriptsize}

The formula of the higher homotopy in Section 4 will give us a cycle, which can be viewed as a generalization of Totaro's cycle.

\end{example}
If we go back to the general case, the similar consideration in Example \ref{gen of Totaro} works. 
\begin{itemize}
\item
Rewrite the relative motive by some double complex whose term are isomorphism to $\mathbb{Q}_{E^n}$ and we get a double complex.
\item
Consider its associated simplex to the above double complex and replace each term by its resolution. Then we get a DG complex in the category of elliptic motives with explicit higher homotopy. 
\item
Based on the results in Section 2.1, we can describe the DG complex as a comodule over the (simplicial) bar construction of the elliptic algebra $\mathcal{E}_{ell}$.
\end{itemize}


\bigskip


\end{document}